\documentclass[times,doublespace,12pts]{amsart}

\usepackage{amsmath, amssymb, amsthm}
\usepackage{mathrsfs}
\usepackage[colorlinks=true,allcolors=blue]{hyperref}
\usepackage[capitalise,nameinlink]{cleveref}

\newtheorem{lemma}{Lemma}
\newtheorem{proposition}{Proposition}
\newtheorem{theorem}{Theorem}
\newtheorem{remark}{Remark}
\newtheorem{corollary}{Corollary}
\newtheorem{question}{Question}

\newcommand{\Q}{\mathbb Q}
\newcommand{\Z}{\mathbb Z}
\newcommand{\Ql}{\mathbb Q_\ell}

\newcommand{\End}{\operatorname{End}}

\newcommand{\Sp}{\operatorname{Sp}}

\title{A Note on Jacobians with Prescribed Factors}
\author{Congling Qiu}
 
\begin{document}

 \begin{abstract}
Over an infinite field, we prove a refinement of Matsusaka's theorem in
which the complementary isogeny factor can be made absolutely simple of
arbitrarily large dimension. We also discuss related questions around relative
simplicity and the isogeny Schottky problem.
\end{abstract}
\maketitle
\section*{Introduction}

A classical theorem of Matsusaka says that every abelian variety is an isogeny
factor of the Jacobian of some smooth projective curve. In this note, we prove the following
refinement.

\begin{theorem}\label{thm:main}
Let \(k\) be an infinite field, and let \(A/k\) be an abelian variety of
dimension \(g\ge 2\). For every integer \(N>0\), there exists a smooth
projective curve \(C/k\) contained in \(A\) such that \(A\) is an isogeny factor
of \(J_C\), and the complementary isogeny factor is absolutely simple of
dimension \(>N\).
\end{theorem}

It is also well known that a very general Jacobian is simple. The theorem suggests a
relative analogue in which one fixes an isogeny factor \(A\): in a sufficiently
large family of Jacobians admitting \(A\) as an isogeny factor, should the
complementary factor be generically simple? See
\Cref{sec:A-special-question} for a precise formulation and for positive
answers in low genus when \(\dim A=1\). In a different direction, see
\Cref{prop:elliptic-complement-simple-small-genus} for positive examples when
\(\dim A\geq 2\).

The theorem also motivates a Grothendieck-group version of the isogeny Schottky
problem, discussed in \Cref{sec:Grothendieck-Schottky}.

As an immediate application of \Cref{thm:main}, we answer the following question of
Poonen, which we learnt from W. Zhang.
\begin{corollary}\label{main:cor}
Let \(k\) be an infinite field, and let \(A\) and \(B\) be abelian varieties
over \(k\) with no common simple isogeny factor over \(k\).  Then there exists a smooth projective
curve \(C/k\) such that \(A\) is an isogeny factor of \(J_C\), while \(J_C\)
has no common simple isogeny factor with \(B\) over \(k\).
\end{corollary}

Indeed, if \(\dim A\ge 2\), apply the theorem to \(A\) with
\(N=\dim B\). If \(\dim A=1\), apply the theorem instead to \(A^2\).  

The statement of \Cref{main:cor} might be   known to experts, although
we have not found a definite reference.  Over \(\mathbb C\), a closely related
isogeny statement is mentioned in the discussion following
\cite[Corollary 1.4]{FS}.  By contrast, \cite[Corollary 1.4]{FS} shows that
the analogous statement with actual direct factors in the category of abelian
varieties is false for a general abelian variety over \(\mathbb C\).

The proof of \Cref{thm:main} is by a monodromy argument for hyperplane sections, which is straightforward at
least when \(k\) is not algebraic over   a finite field. The required monodromy was established by Deligne \cite{WeilII}.  (We also learnt that Petrov had proposed a similar solution to  \Cref{main:cor} recently.)
When \(k\) is  a finite field, we expect the theorem and thus the corollary remain true as stated, though this does not seem to follow from the present approach.    We also discuss the case \(\dim A=1\) in the end of \Cref{sec:Absolute Simplicity}, after the proof of the main theorem.
 \subsection*{Acknowledgment}
The author thanks W. Zhang for helpful communications.

The author used ChatGPT and Gemini as writing assistants in preparing parts of
the manuscript, including suggestions for exposition and draft arguments. The
author checked and revised all such output and takes full responsibility for
the mathematical content.
 
\section{Monodromy of Successive Complete Intersections}

This section establishes the maximal symplectic monodromy for successive smooth complete intersections.
This serves as a minor but natural generalization of the classical statement for Lefschetz pencils in Deligne \cite{WeilII}.

Let \(k\) be a field and \(\ell\) a prime number different from the characteristic of \(k\).
Let \(X\) be a smooth projective variety of dimension \(n+r\) over \(k\). Let \(S_1,\dots,S_{r-1}\) be a fixed flag of smooth complete intersections such that
\[
X\supset S_1\supset S_2\supset \dots \supset S_{r-1},
\]
where \(\dim(S_i)=n+r-i\). Let \(S=S_{r-1}\), so that \(\dim S=n+1\).

By the Weak Lefschetz Theorem, after base change to \(\bar k\), the restriction maps
\[
H^n(X_{\bar k},\Ql)\xrightarrow{\sim}
H^n((S_1)_{\bar k},\Ql)\xrightarrow{\sim}\dots\xrightarrow{\sim}
H^n(S_{\bar k},\Ql)
\]
are isomorphisms. Also, for a smooth hypersurface section \(Y\subset S\), the
restriction map
\[
H^n(S_{\bar k},\Ql)\hookrightarrow H^n(Y_{\bar k},\Ql)
\]
is injective. Thus we obtain a canonical inclusion
\[
H^n(X_{\bar k},\Ql)\hookrightarrow H^n(Y_{\bar k},\Ql).
\]

\begin{lemma}\label{lem:l-adic-vanishing-monodromy}
Let \(U\subset \mathbb P^1_k\) be the smooth locus of a Lefschetz pencil of
hypersurface sections \(Y_t\subset S\). Assume \(n\) is odd. Then the geometric
algebraic monodromy group, defined as the Zariski closure of the image of
\(\pi_1(U_{\bar k})\) acting on
\[
H^n_{\mathrm{van}}(Y_{t,\bar k},\Ql)
:=
H^n(Y_{t,\bar k},\Ql)/
\operatorname{Im}\bigl(H^n(X_{\bar k},\Ql)\to H^n(Y_{t,\bar k},\Ql)\bigr),
\]
is
\(
\Sp\bigl(H^n_{\mathrm{van}}(Y_t,\Ql)\bigr)
\)
with respect to the intersection pairing.
\end{lemma}

\begin{proof}
Since \(n\) is odd, this follows from Deligne's big \(\ell\)-adic monodromy
\cite[Thm. 4.4.1 and Lem. 4.4.2a]{WeilII}, applied over \(\bar k\).
\end{proof}

The same statement holds with \(\mathbb F_\ell\)-coefficients,  by the Weak Lefschetz Theorem with finite coefficients
\cite[Exp. XIV, Cor. 3.3]{SGA4half}.

\begin{lemma}\label{lem:mod-l-vanishing-monodromy}
Keep the notation of the preceding lemma. Then, for all sufficiently large primes
\(\ell\neq \operatorname{char}(k)\), the geometric monodromy action on
\[
H^n_{\mathrm{van}}(Y_t,\mathbb F_\ell)
:=
H^n(Y_{t,\bar k},\mathbb F_\ell)/
\operatorname{Im}\bigl(H^n(X_{\bar k},\mathbb F_\ell)
\to H^n(Y_{t,\bar k},\mathbb F_\ell)\bigr)
\]
has image
\(
\Sp\bigl(H^n_{\mathrm{van}}(Y_t,\mathbb F_\ell)\bigr)
\)
with respect to the intersection pairing.
\end{lemma}

\begin{proof}
By Gabber's torsion-freeness theorem \cite{Gab}, after excluding finitely many
primes \(\ell\), the relevant \(\mathbb Z_\ell\)-cohomology groups are
torsion-free. Then one can apply a version of the big mod-\(\ell\) monodromy
theorem for Lefschetz pencils, coming from the standard Katz--Deligne theory and
its refinements by Hall; see for example \cite[Theorem 3.5]{RSV}.
\end{proof}

\section{Absolute Simplicity of the Vanishing Jacobian}\label{sec:Absolute Simplicity}
We now specialize \Cref{lem:l-adic-vanishing-monodromy} and \Cref{lem:mod-l-vanishing-monodromy} to the case of curves
embedded in an abelian variety.
Let \(A\) be an abelian variety of dimension \(g\ge 2\) over \(k\). Set \(X=A\). Choose a projective embedding \(A\hookrightarrow \mathbb P^N_k\). By Bertini
over infinite fields, and by Poonen's Bertini theorem \cite{Poo04} over finite fields, after
allowing hypersurfaces of sufficiently large degree, we may choose such a smooth
complete-intersection flag
\[
    A\supset S_1\supset \cdots \supset S_{g-2}=S
\]
defined over \(k\).

For the generic  curve \(C_\eta\subset S\) of a Lefschetz pencil,   \(J_{C_\eta}\) splits up to isogeny as
\[
J_{C_\eta}\sim A_{k(\eta)}\times A_{\mathrm{van}},
\]
where, after fixing a geometric generic point
\(\bar\eta\to\eta\), 
\[
H^1(A_{\mathrm{van},\bar \eta},\Ql)\cong H^1_{\mathrm{van}}(C_{\bar \eta},\Ql).
\]

\begin{proposition}\label{prop:vanishing-jacobian-abs-simple}
The abelian variety \(A_{\mathrm{van}}\) is absolutely simple.
\end{proposition}

\begin{proof}
Set \(D=\End^0(A_{\mathrm{van},\bar \eta})\). The algebra
\(D\otimes_{\Q}\Ql\) acts faithfully on
\[H^1(A_{\mathrm{van},\bar \eta},\Ql)\simeq H^1_{\mathrm{van}}(C_{\bar\eta},\Ql).\] After replacing \(U\) by a finite \'etale cover, every element of \(D\) spreads out
and hence commutes with a finite-index subgroup \(\Gamma\subset \pi_1(U_{\bar k})\).
The Zariski closure of the image of \(\Gamma\) has the same identity component as
the Zariski closure of the image of \(\pi_1(U_{\bar k})\). Therefore this identity
component is \(\Sp(H^1_{\mathrm{van}}(C_{\bar\eta},\Ql))\) by \Cref{lem:l-adic-vanishing-monodromy}.
Hence \(D\otimes_{\Q}\Ql\) lies in the centralizer of this symplectic group. By
Schur's lemma, this centralizer is \(\Ql\). Thus \(D=\Q\), and \(A_{\mathrm{van}}\)
is absolutely simple.
\end{proof}

\begin{proposition}[Hilbertian specialization of absolute simplicity]\label{cor:hilbertian-specialization-abs-simple}
Assume that \(k\) is Hilbertian. Then there exist infinitely many \(t\in U(k)\)
such that \(A_{\mathrm{van},t}\) is absolutely simple.
\end{proposition}

\begin{proof}
By \Cref{lem:l-adic-vanishing-monodromy} and Serre's Hilbert irreducibility theorem \cite[Chapter 9]{Ser}, there exist infinitely many
\(t\in U(k)\) such that the Zariski closure of the specialized Galois image on
\(
H^1((A_{\mathrm{van},t})_{\bar k},\Ql)
\)
contains
\(\Sp(H^1((A_{\mathrm{van},t})_{\bar k},\Ql))\). Fix such a \(t\), and set
\(D_t:=\End^0((A_{\mathrm{van},t})_{\bar k})\). Every element of \(D_t\) is
defined over some finite extension \(L/k\), so \(D_t\otimes_{\Q}\Ql\) commutes
with the image of the absolute Galois group of \(L\). Since this image still has
Zariski closure containing the connected group
\(\Sp(H^1((A_{\mathrm{van},t})_{\bar k},\Ql))\), we get
\[
D_t\otimes_{\Q}\Ql
\subseteq
\End_{\Sp}\bigl(H^1((A_{\mathrm{van},t})_{\bar k},\Ql)\bigr)
=
\Ql.
\]
Thus \(D_t=\Q\), and \(A_{\mathrm{van},t}\) is absolutely simple.
\end{proof}

\begin{proposition}[Finite-field specialization of absolute simplicity]\label{prop:finite-field-specialization-abs-simple}
Assume that \(k=\mathbb F_q\). Then there exist infinitely many closed points
\(t\in |U|\) such that \((A_{\mathrm{van},t})_{\bar k}\) is absolutely simple. Moreover, if \(k'\subset \bar k\) is any infinite extension of \(k\), the closed
points \(t\) may be chosen with \(k(t)\subset k'\).

\end{proposition}

\begin{proof}
For a closed point \(t\in |U|\), define
\[
P_t(T):=\det\bigl(1-T\mathrm{Frob}_t\mid
H^1((A_{\mathrm{van},t})_{\bar k},\Ql)\bigr).
\]
For \(r\ge 1\), let \(P_t^{(r)}(T)\) denote the polynomial whose roots are the
\(r\)-th powers of the roots of \(P_t(T)\). Equivalently, if
\(k(t)=\mathbb F_{q^m}\), then \(P_t^{(r)}(T)\) is the Frobenius polynomial of
\(A_{\mathrm{van},t}\) after base change to \(\mathbb F_{q^{mr}}\).

By \Cref{lem:mod-l-vanishing-monodromy}, for all sufficiently large
\(\ell\), the geometric mod-\(\ell\) monodromy of the vanishing system is the
full symplectic group. Hence
Chavdarov's theorem \cite[Theorem 2.1]{Cha} applies to the vanishing compatible
system. Therefore there exist infinitely many closed points \(t\in |U|\) such
that \(P_t^{(r)}(T)\) is irreducible over \(\Q\) for every \(r\ge 1\).
For such \(t\), the abelian variety \(A_{\mathrm{van},t}\) is simple over every
finite extension \(\mathbb F_{q^{mr}}\) of \(k(t)\), hence
\((A_{\mathrm{van},t})_{\bar k}\) is absolutely simple.

Let \(k'\subset \bar k\) be an infinite extension of \(k=\mathbb F_q\). Choose
an increasing sequence of finite subfields
\[
    \mathbb F_{q^{n_1}}\subset \mathbb F_{q^{n_2}}\subset \cdots \subset k',
    \qquad n_j\to \infty .
\]
We claim that Chavdarov's argument applies along this sequence. Indeed, in the proof of \cite[Theorem 2.1]{Cha}, after fixing the finitely many
auxiliary primes, the application of the equidistribution theorem
\cite[Theorem 4.1]{Cha} gives an error term \(O(q^{-n/2})\), with the implicit
constant independent of \(n\). Hence the proportion of points
\(u\in U(\mathbb F_{q^{n_j}})\) for which \(P_u^{(r)}(T)\) is irreducible over
\(\Q\) for every \(r\ge 1\) tends to \(1\) as \(j\to\infty\).
By Lang--Weil,
\(
    \# U(\mathbb F_{q^{n_j}})\sim q^{n_j\dim U}.
\)
Thus the number of such good points in \(U(\mathbb F_{q^{n_j}})\) tends to
infinity. On the other hand, a fixed closed point \(t\in |U|\) of degree \(d\)
contributes points to \(U(\mathbb F_{q^{n_j}})\) only when \(d\mid n_j\), and
then contributes at most \(d\) points. Therefore any finite set of closed points
contributes a uniformly bounded number of points to all
\(U(\mathbb F_{q^{n_j}})\). Since the number of good points tends to infinity,
the corresponding closed points cannot be finite in number. Hence there are
infinitely many closed points \(t\in |U|\) with \(k(t)\subset k'\) such that
\(P_t^{(r)}(T)\) is irreducible over \(\Q\) for every \(r\ge 1\).
\end{proof}

\begin{proof}[Proof of \Cref{thm:main}]
We first choose the degree of the hypersurface section sufficiently large, so
that the complementary factor has dimension \(>N\). This condition is
preserved under all specializations considered below.

If \(k\) is not algebraic over a finite field, then \(k\) is Hilbertian after the
standard reduction to a finitely generated subfield, and the result follows from
\Cref{cor:hilbertian-specialization-abs-simple}.

It remains to consider the case where \(k\) is algebraic over a finite field.
All the data involved in the construction are defined over a finite subfield
\(k_0=\mathbb F_q\subset k\), after replacing \(k_0\) by a finite extension if
necessary. If \(k\) is finite, then
\Cref{prop:finite-field-specialization-abs-simple} gives a closed point
\(t\in |U|\); after replacing \(k\) by the finite residue field \(k(t)\), the
corresponding curve gives the desired example.
Assume now that \(k\) is infinite. Applying the strengthened form of
\Cref{prop:finite-field-specialization-abs-simple} to the infinite subfield
\(k\subset \bar k_0\), we may choose \(t\in |U|\) with \(k(t)\subset k\) such
that the corresponding vanishing factor is absolutely simple over \(\bar k\).
Thus the specialization is already defined over \(k\), and the corresponding
curve \(C/k\) has \(A\) as an isogeny factor of \(J_C\) with absolutely simple
complement.
\end{proof}

\begin{remark}
When \(\dim A=1\), there is another natural approach. One can use ramified
double covers \(C\to A\), in which case the complementary factor of \(A\) in
\(J_C\) is the associated Prym variety. Thus one could try to prove the desired
simplicity by proving big monodromy on the anti-invariant cohomology in this
family. We have not found this monodromy statement explicitly in the literature.

Within the same ramified double cover construction, another possible route is
through the Néron--Severi group: rank 1 of this group rules out a
nontrivial isogeny decomposition. Over \(\mathbb C\), this route was carried out
by Biswas and Paranjape in \cite{BP02}, who prove that a general Prym variety
has Néron--Severi group \(\Z\).
\end{remark}
 \section{A Question on Maximal \(A\)-Special Subvarieties of \(\mathcal M_g\)}\label{sec:A-special-question}

Let \(k\) be a field, and let \(A/k\) be an abelian variety of dimension
\(a<g\). Let \(\mathcal M_g\) denote the moduli stack of smooth projective
curves of genus \(g\) over \(k\). We say that an irreducible \(k\)-subvariety
\(V\subset \mathcal M_g\) is \(A\)-special if, for the generic curve
\(C_{\eta}/k(V)\) corresponding to \(V\), the abelian variety \(A_{k(V)}\)
is an isogeny factor of \(J_{C_{\eta}}\). Moreover, by
\Cref{prop:HA-union-Aspecial} below, for every point \(x\in V\), the abelian
variety \(A_{k(x)}\) is an isogeny factor of \(J_{C_x}\). Let
\(\mathscr C_A(g)\) be the set of irreducible \(A\)-special \(k\)-subvarieties
of \(\mathcal M_g\). We call \(V\in \mathscr C_A(g)\) maximal if there is no
\(W\in \mathscr C_A(g)\) such that \(V\subsetneq W\).

\begin{question}\label{ques:simple}
Let \(V\in \mathscr C_A(g)\) be maximal. Must the complementary isogeny factor
of \(A_{k(V)}\) in \(J_{C_{\eta}}\) be geometrically simple? More strongly,
when \(k=\mathbb C\), is the complementary isogeny factor simple for a very
general point \(x\in V\)?
\end{question}

We have a positive answer to \Cref{ques:simple} in some simple cases below.
Before that, let us first explain the notion of \(A\)-specialness.

The notion of \(A\)-specialness is motivated by the corresponding special loci
in the moduli space \(\mathcal A_g\) of principally polarized abelian
varieties. Choose a principally polarized abelian variety \((A_0,\lambda_0)\)
whose underlying abelian variety is isogenous to \(A\). Consider the product
locus \(\{A_0\}\times \mathcal A_{g-a}\subset \mathcal A_g\). Define
\(Z_A\subset \mathcal A_g\) to be the union of all Hecke translates of
\(A_0\times \mathcal A_{g-a}\). Each such Hecke translate is a weakly special
subvariety of \(\mathcal A_g\). It is special when \(A\) is CM. Moreover, this
definition is independent of the choice of \((A_0,\lambda_0)\).

Let \(t:\mathcal M_g\to \mathcal A_g\) be the Torelli morphism. Define the
\(A\)-special locus
\[
    H_A(g):=t^{-1}(Z_A):=\bigcup_i t^{-1}(Z_i).
\]
Then \(H_A(g)\) is a countable union of closed algebraic subvarieties of
\(\mathcal M_g\).

\begin{proposition}\label{prop:HA-union-Aspecial}
The \(A\)-special locus \(H_A(g)\) is the union of \(A\)-special subvarieties
of \(\mathcal M_g\).
\end{proposition}

\begin{proof}
The proof is standard. The key point is the following pointwise containment.
Let \([C]\in \mathcal M_g\) be a geometric point over an algebraically closed
field \(\Omega\), and suppose that \(J_C\) has an isogeny factor isogenous to
\(A_\Omega\). Then \([C]\in H_A(g)\).

Indeed, \(J_C\) is isogenous, as a principally polarized abelian variety up to
polarized quasi-isogeny, to a product \(A_1\times B_1\), where \(A_1\) is
isogenous to \(A_\Omega\). The polarizations induced on the two factors need
not be principal, but by the standard descent argument for polarizations one
quotients by maximal isotropic subgroup schemes of the kernels to obtain
principally polarized varieties \(A'\sim A_\Omega\) and \(B'\). See, for
example, \cite[§23, Theorem~4 and Corollary~1]{Mum}. Hence \([C]\in H_A(g)\).
\end{proof}

We have a positive answer to \Cref{ques:simple} in the following simple cases.
The proof takes the rest of this section.

\begin{proposition}\label{prop:elliptic-complement-simple-small-genus}
Assume \(k=\mathbb C\). Let \(E/k\) be an elliptic curve, and let \(V\) be a
maximal \(E\)-special subvariety. If \(g\le 14\), then for a very general point
\([C]\in V\), the complementary isogeny factor of \(E\) in \(J_C\) is
geometrically simple.
\end{proposition}

\begin{lemma}[{\cite[Corollary E]{LMP}}]\label{lem:HA-bound}
Assume \(k=\mathbb C\), \(a=\dim A>1\), and \(g>2a-2\). Then every
\(A\)-special subvariety \(V\subset \mathcal M_g\) satisfies
\[
    \dim V\le 2g-4a+2 .
\]
\end{lemma}

On the other hand, when \(a=1\), by the Riemann--Hurwitz formula, one has the
following standard dimension count.

\begin{lemma}\label{lem:elliptic-factor-dim}
Let \(E/k\) be an elliptic curve. Then every maximal \(E\)-special subvariety
\(V\subset \mathcal M_g\) satisfies \(\dim V=2g-3\).
\end{lemma}

\begin{proof}[Proof of \Cref{prop:elliptic-complement-simple-small-genus}]
By \Cref{lem:elliptic-factor-dim}, we have \(\dim V=2g-3\). We show that the
locus of points \([C]\in V\) for which the complementary factor of \(E\) in
\(J_C\) is not simple is contained in a countable union of proper closed
subvarieties of \(V\).

Suppose that, for some point \([C]\in V\), the complementary factor is not
simple. Then, up to isogeny, \(J_C\sim E\times P\), with
\(P\sim B_1\times B_2\), where \(b_i:=\dim B_i\ge 1\) and
\(b_1+b_2=g-1\). After relabeling, assume \(b_1\le b_2\). Set
\(a:=\dim(E\times B_1)=1+b_1\). Then \(a>1\), and
\(g-1=b_1+b_2\ge 2b_1\), so \(g\ge 2b_1+1>2b_1=2a-2\). Thus
\Cref{lem:HA-bound} applies to the fixed abelian variety \(A=E\times B_1\).

For each fixed \(B_1\), the locus of curves whose Jacobians have
\(E\times B_1\) as an isogeny factor has dimension at most
\[
    2g-4a+2=2g-4(b_1+1)+2=2g-4b_1-2.
\]
Now allow \(B_1\) to vary. The possible \(B_1\)'s of dimension \(b_1\) vary in
\(\mathcal A_{b_1}\), which has dimension \(b_1(b_1+1)/2\). Isogeny choices
are countable and do not affect dimension. Hence the locus of curves
\([C]\in \mathcal M_g\) for which \(J_C\) has an isogeny factor isogenous to
\(E\times B_1\), for some \(B_1\) of dimension \(b_1\), has dimension at most
\[
    \frac{b_1(b_1+1)}{2}+2g-4b_1-2.
\]
Comparing with \(\dim V=2g-3\), it is enough to have
\[
    \frac{b_1(b_1+1)}{2}+2g-4b_1-2<2g-3,
\]
or equivalently \(b_1^2-7b_1+2<0\). This holds for every integer
\(1\le b_1\le 6\).

If \(g\le 14\), then \(b_1\le \lfloor (g-1)/2\rfloor\le 6\). Therefore, for
every possible \(b_1\), the corresponding splitting locus has dimension
strictly smaller than \(\dim V=2g-3\). Thus its intersection with the
irreducible variety \(V\) is a proper closed subvariety of \(V\), up to taking
the relevant irreducible components. Since there are only countably many
Hecke/isogeny choices and only finitely many possible values of \(b_1\), the
locus in \(V\) where the complementary factor is not simple is contained in a
countable union of proper closed subvarieties.

Hence for a very general point \([C]\in V\), the complementary isogeny factor
of \(E\) in \(J_C\) is simple. Since \(k=\mathbb C\), this is the same as being
geometrically simple.
\end{proof}
  
\section{A Grothendieck group version of the isogeny Schottky problem}\label{sec:Grothendieck-Schottky}

The usual Schottky problem asks whether a principally polarized abelian variety
lies in the Jacobian locus. Its isogeny variant asks what remains true after
one forgets the polarization and works only up to isogeny. In this direction,
it is natural to consider the quotient
\[
G_{\mathrm{Jac}}(k)
:=
K_0(\mathrm{AV}^{\mathrm{isog}}_k)\Big/
\left\langle [J_C] \;:\; C/k \text{ a smooth projective curve} \right\rangle .
\]
For an abelian variety \(A/k\), the class \([A]\in G_{\mathrm{Jac}}(k)\) is \(0\)
  if and only if there are smooth projective curves
\(C_1,\dots,C_r\) and \(D_1,\dots,D_s\) over \(k\), and an integer \(n>0\),
such that
\[
    A\times J_{D_1}\times\cdots\times J_{D_s}
    \sim
    J_{C_1}\times\cdots\times J_{C_r}
\]
over \(k\).  

Theorem 1.2 of \cite{FS} shows that if \(A\) is either the intermediate
Jacobian of a very general cubic threefold, or a very general principally
polarized abelian variety of dimension at least \(4\), then no positive power
\(A^n\) is isogenous to a product of Jacobians. This suggests the following
stronger stabilized question in the Grothendieck group.

\begin{question}\label{ques:GJac-nontorsion}
Let \(A/\mathbb C\) be a very general principally polarized abelian variety of
dimension \(g\ge 4\). Is the class \([A]\in G_{\mathrm{Jac}}(\mathbb C)\)
non-torsion?
\end{question}

We expect \(G_{\mathrm{Jac}}(\mathbb C)\) to be very large. However, there
does not seem to be an easy way to put a reasonable dimension on this group.
One can instead ask a counting question.

\begin{question}\label{ques:GJac-uncountable}
For some \(g\ge 3\), is the image of the natural map
\[
    \mathcal A_g(\mathbb C)\longrightarrow G_{\mathrm{Jac}}(\mathbb C),
    \qquad
    A\longmapsto [A],
\]
uncountable?
\end{question}

The analogous counting problem over finite fields also seems intriguing. Since
the isogeny category over a finite field is much more arithmetic and discrete,
one may ask for quantitative variants of \Cref{ques:GJac-uncountable} over
finite fields. We leave this as a possible direction for future work.

The main theorem of this note gives a weaker but useful structural statement.

\begin{proposition}\label{prop:GJac-simple-representative}
Assume that \(k\) is not a finite field. Let \(A/k\) be an abelian variety.
Then the class of \(A\) in \(G_{\mathrm{Jac}}(k)\) is either \(0\), or is equal
to the class of an absolutely simple abelian variety over \(k\).
\end{proposition}

\begin{proof}
If \(\dim A=1\), then \(A\) is an elliptic curve, hence \(A=J_A\), and so
\([A]=0\) in \(G_{\mathrm{Jac}}(k)\). Thus we may assume that either
\([A]=0\), or \(\dim A\ge 2\).

In the latter case, apply \Cref{thm:main} to \(A\). This gives an absolutely
simple abelian variety \(S/k\) such that \([A]=-[S]\). Applying
\Cref{thm:main} again to \(S\), we get an absolutely simple abelian variety
\(T/k\) such that \([S]=-[T]\). Therefore \([A]=[T]\) in
\(G_{\mathrm{Jac}}(k)\).
\end{proof}

We record the following simple but useful observation.

\begin{lemma}\label{lem:complements-same-class}
Let \(A/k\) be an abelian variety. For \(i=1,2\), let \(C_i/k\) be a smooth
projective curve such that \(J(C_i)\sim A\times P_i\). Then
\([P_1]=[P_2]\) in \(G_{\mathrm{Jac}}(k)\).
\end{lemma}

The lemma implies that, along an \(A\)-special family of curves, all isogenous complementary
factors of \(A\) in the Jacobians  have the same class in \(G_{\mathrm{Jac}}\). For instance, take
\(k=\mathbb C\) and let \(E/\mathbb C\) be an elliptic curve. By
\Cref{lem:elliptic-factor-dim}, a maximal \(E\)-special subvariety
\(V\subset \mathcal M_g\) has dimension \(2g-3\). After choosing
polarizations on the complementary factors, they give a \(2g-3\)-dimensional
family in \(\mathcal A_{g-1}\), while their image in
\(G_{\mathrm{Jac}}(\mathbb C)\) is a single point. Moreover, if \(g\le 14\),
then by \Cref{prop:elliptic-complement-simple-small-genus}, a very general
member of this family is geometrically simple.

In particular, when \(g=4\), this gives a \(5\)-dimensional family of abelian
threefolds whose image in \(G_{\mathrm{Jac}}(\mathbb C)\) is a single point.
This raises the question of what the image of the \(6\)-dimensional
\(\mathcal A_3\) in \(G_{\mathrm{Jac}}(\mathbb C)\) looks like. More
generally, it is not clear to us how to put any reasonable geometric structure on
\(G_{\mathrm{Jac}}(\mathbb C)\).

\end{document}